\documentclass[10pt]{amsart}
\usepackage{latexsym,amsmath,amssymb,graphicx}
\usepackage[bf, small]{caption}
\usepackage[all,web]{xy}
\usepackage{color}
\usepackage{hyperref}
\usepackage[utf8]{inputenc}
\usepackage{todonotes}
\usepackage{enumitem}
\def\isom{\buildrel\sim\over\longrightarrow}
\numberwithin{equation}{section}

\usepackage{xcolor}

\def\xyellowspace{%
  \sbox0{\colorbox{yellow}{\strut\ }}
  \dimen0=\wd0\relax
  \hskip0pt\cleaders\box0\hskip\dimen0\hskip0pt}

\begingroup
\catcode`\ =\active%
\gdef\makeyellowspace{\let \xyellowspace\catcode`\ =\active}%
\endgroup

\def\?#1{\colorbox{yellow}{\strut#1}}

\def\urlfont{\DeclareFontFamily{OT1}{cmtt}{\hyphenchar\font='057}
              \normalfont\ttfamily \hyphenpenalty=10000}

\DeclareFontFamily{OT1}{rsfs10}{}
\DeclareFontShape{OT1}{rsfs10}{m}{n}{ <-> rsfs10 }{}
\DeclareMathAlphabet{\mathscript}{OT1}{rsfs10}{m}{n}

\DeclareMathOperator{\Hom}{Hom}     
\DeclareMathOperator{\Pic}{Pic}     
\DeclareMathOperator{\rk}{rk}       
\DeclareMathOperator{\Mov}{Mov}     
\DeclareMathOperator{\Relint}{Relint}  
\DeclareMathOperator{\Min}{Min}     

\DeclareMathOperator{\nM}{\overline{Min}}
\DeclareMathOperator{\MnM}{M}

\newcommand{\RR}{\mathbb{R}}

\newcommand{\ZZ}{\mathbb{Z}}

\newcommand{\KK}{\mathbb{K}}
\newcommand{\QQ}{\mathbb{Q}}

\title[Toric varieties and Gr\"obner bases]{Toric varieties and Gr\"obner bases:\\ the complete $\QQ$-factorial case}

\author[M. Rossi and L.Terracini]{Michele Rossi and Lea Terracini}

\date{\today}

\address{Dipartimento di Matematica, Universit\`a di Torino,
via Carlo Alberto 10, 10123 Torino} \email{michele.rossi@unito.it,
lea.terracini@unito.it}

\thanks{The first author was partially supported by the I.N.D.A.M. as a member of the G.N.S.A.G.A.}

\def \Si{\Sigma }

\def \aa{\mathbf{a}}
\def \bb{\mathbf{b}}
\def \cc{\mathbf{c}}
\def \e{\mathbf{e}}

\def \q{\mathbf{q}}

\def \u{\mathbf{u}}
\def \v{\mathbf{v}}
\def \n{\mathbf{n}}
\def \m{\mathbf{m}}
\def \w{\mathbf{w}}
\def \t{\mathbf{t}}
\def \z{\mathbf{z}}
\def \x{\mathbf{x}}

\def \y{\mathbf{y}}
\def \1{\mathbf{1}}
\def \0{\mathbf{0}}

\def\p2{\mathbb{P}^2}
\def\p3{\mathbb{P}^3}
\def\p4{\mathbb{P}^4}
\def\r{\mathbf{r}}

\def\coL{\mathcal{L}}
\def\coU{\mathcal{U}}
\def\coV{\mathcal{V}}
\def\coW{\mathcal{W}}

\def\rk{\operatorname{rk}}

\def\GL{\operatorname{GL}}

\def\Z{\mathbb{Z}}

\def\K{\mathbb{K}}

\def\R{\mathbb{R}}

\def\Q{\mathbb{Q}}
\def\N{\mathbb{N}}

\def\V{\mathbb{V}}

\def\F{\mathcal{F}}

\def\SF{\mathcal{SF}}
\def\PSF{\mathcal{PSF}}

\def\gkz{\mathcal{Q}}

\def\G{\mathcal{G}}

\def\Ls{\mathcal{L}}

\theoremstyle{plain}
\newtheorem{theorem}{Theorem}
\newtheorem{proposition}{Proposition}

\newtheorem{thm-def}{Theorem--Definition}
\newtheorem{corollary}{Corollary}
\newtheorem{conjecture}{Conjecture}
\newtheorem{lemma}{Lemma}

\newtheorem*{a-proposition}{Proposition}

\theoremstyle{remark}
\newtheorem{remark}{Remark}

\theoremstyle{definition}
\newtheorem{definition}{Definition}

\newtheorem*{step I}{Step I}
\newtheorem*{step II}{Step II}
\newtheorem*{step III}{Step III}
\newtheorem*{step IV}{Step IV}

\begin{document}


\begin{abstract}
We present two algorithms  determining all the complete and simplicial fans admitting a fixed non-degenerate set of vectors $V$ as generators of their 1-skeleton. The interplay of the two algorithms  allows us to discerning if the associated toric varieties admit a projective embedding, in principle for any values  of dimension and Picard number.
The first algorithm is slower than the second one, but it computes all complete and simplicial fans supported by $V$ and lead us to formulate a topological-combinatoric conjecture about the definition of a fan.

On the other hand, we adapt the Sturmfels' arguments on the Gr\"obner fan of toric ideals  to our complete case; we give a characterization of the Gr\"obner region and show an explicit correspondence between  Gr\"obner cones and  chambers of the secondary fan. A homogenization procedure of the toric ideal associated to $V$ allows us to employing GFAN and related software in producing our second algorithm. The latter turns out to be much faster than the former, although it can compute only the projective fans supported by $V$. We provide examples and a list of open problems. In particular we give examples of rationally parametrized families of $\Q$-factorial complete toric varieties behaving in opposite way with respect to the dimensional jump of the nef cone over a special fibre.
\end{abstract}



\maketitle

\section*{Introduction}
The main purpose of the study here presented is producing two implemented algorithms aimed to determining all the complete and simplicial fans, admitting a fixed non-degenerate set of vectors as generators of their 1-skeleton. In particular their interplay allows us to discerning if the associated toric varieties admit a projective embedding, that is ample divisors, in particular for higher values ($\geq 4$) of dimension and Picard number.

In fact, it is well known that a complete toric variety may not be projective. This cannot happen for toric varieties of dimension $\leq 2$ \cite[\S\,8, Prop.\,8.1]{MO}, but for higher dimension it has been shown by several examples, the first of which was given by M.\,Demazure \cite{Demazure}. P.\,Kleinschmidt and B.\,Sturmfels \cite{Kleinschmidt-Sturmfels} proved that, for Picard number (in the following also called the \emph{rank}) $r\leq 3$, smooth and complete toric varieties are projective in every dimension, that is they have to admit ample divisors. This result cannot be extended to higher values of the rank, as shown by a famous example given by T.~Oda \cite[p.84]{Oda}, who presented a smooth complete 3--dimensional toric variety $X$ of rank $r=4$, whose nef cone has dimension 2: therefore $X$ admits non--trivial numerically effective classes (among which the anti--canonical one) but does not admit any ample class.

\noindent When dropping the smoothness hypothesis, Kleinschmidt--Sturmfels bound does no longer hold even for $\Q$--factorial singularities: a counterexample has been given by F.\,Berchtold and J.\,Hausen \cite[Ex.\,10.2]{Berchtold-Hausen} who provided a $\Q$--factorial complete 3--dimensional toric variety $X$ of rank $r=3$ whose nef cone is 1-dimensional and generated by the anti--canonical class. This example is actually a divisorial contraction of the Oda's example, as it can be obtained by suppressing a fan generator. As the Oda's one, the Berchtold-Hausen example is still a sharp counterexample, as $\Q$-factorial complete toric varieties of rank $r\leq 2$ turn out to be projective in any dimension, as recently proved in \cite{RT-r2proj} by the authors of the present paper.

In the literature, further examples of non-projective complete toric varieties are known. In particular O.~Fujino and S.~Payne \cite{FP} provided an example of a smooth toric threefold of rank $r=5$ without any non-trivial numerically effective divisor, that is, whose nef cone is trivial, and showed that such a phenomenon cannot happen, in the smooth -dimensional case, when $r\leq 4$. Moreover an example of a $\Q$-factorial complete 3-dimensional toric variety, of rank $r=3$, whose nef cone is trivial has been recently given by the authors \cite[\S~3]{RT-r2proj}, showing that in the $\Q$-factorial setup that bound reduces to giving $r\leq 2$, when actually the involved varieties have to be projective and so admitting even ample divisors. Again, the last example turns out to be a divisorial contraction of the Fujino-Payne one. Moreover it has been obtained by slightly \emph{deforming}, in a sense better described in the final \S~\ref{ssez:deforming}, the Berchtold-Hausen example, so breaking up its symmetry in such a way that the two non-projective varieties sharing the same nef cone generated by the anti-canonical class \emph{deform}, on the one hand to a projective variety and on the other one to a still worse non-projective one, no more admitting even non-trivial nef divisors.

This fact reveals two interesting evidences.

The first one is that many known examples of complete and non-projective toric varieties are intimately related, showing that, morally, they turn out to be an \emph{ad hoc} variation of the classical Oda's example. We believe this happens because it is very difficult constructing example of this kind for higher values of dimension and rank and maintaining low values of those parameters forces us to work in a too narrow environment.

The second evidence is that we find a gap in the theory, when we try to explain how non-projective varieties arise among complete toric varieties: is it true that a \emph{general} complete toric variety is projective? What is the precise meaning of the word ``general'' in the previous question? Why \emph{slightly deforming} the secondary fan, e.g. by moving a ray of the effective cone, generated by the classes of effective divisors, non-projective varieties can appear and disappear? As stimulating examples in \S\,\ref{ssez:deforming} we present two families, over $\Q$, of complete $\Q$-factorial toric varieties, the first admitting a projective general fibre with a non-projective special fibre and the second admitting general fibre without any non trivial nef divisor and a special fibre whose anticanonical class generates the nef cone. These examples seem to be  representative of apparently opposite phenomena calling for a more general theoretic explanation we are not able to give in this moment.

Both these evidences requires a huge multitude of examples to consider and study, in particular for higher values of dimension and rank, so highlighting the need of a computer-aided approach to this kind of problems.

In the literature there are algorithms computing projective fans (i.e. regular triangulations) but also algorithm computing all kind of triangulations (see for example \cite{DeLoera1996, DeLoera2010, ImaiMasada2002} )
In this paper we propose two algorithms calculating complete $\QQ$-factorial fans over a set of vectors.

The first one, described in Section \ref{sec:cercafan},  computes projective and non projective fans. Although it is quite inefficient, it is theoretically interesting as leading us to the topological-combinatoric conjecture \ref{congettura}, about the definition of a fan. We performed a Maple implementation of this algorithm, making it compatible with packages like \texttt{Convex} by M.\,Franz \cite{convex} and \texttt{MDSpackage} by S.\,Keicher and J.\,Hausen \cite{MDSpackage}. These connections reveal to be quite useful for several applications. The first one in computing  movable and nef cones associated to the computed fans and so detecting if the associated toric variety is projective or not. The second one to study possible embedded Mori Dream spaces.

The second algorithm uses Gr\"obner bases of toric ideals, and  the fact, established by Sturmfels \cite[Prop. 8.15]{Sturmfels1996}, that the Gr\"obner fan of a toric ideal   refines the secondary fan associated with the corresponding configuration of vectors. Whilst most of the literature on this subject  deals with homogeneous toric ideals, here we focus our attention on toric ideals arising from a complete configuration; these ideals cannot be homogeneous. We describe the Gr\"obner region and translate the problem of calculating fans in a problem of linear programming, as done in \cite{SturmfelsThomas97} in the homogeneous case. It turns out a surjective map (Corollary \ref{cor:suriezione}) from a  set of initial ideals of the toric ideal and the set of simplicial projective fans having as $1$-skeleton the rays generated by all the vectors in the configuration.   This correspondence in general is not injective, but initial ideals associated to the same fan result to have the same radical.\\
The above theoretical results, mainly concentrated in \S\,\ref{sez:toricideal}, are probably the most original contribution of the research here presented, and  allow us to formulate, in Section \ref{sec:algorithm}, an algorithm which determines all projective complete simplicial fans with a given $1$-skeleton, based on the computation of the Gr\"obner fan of the associated toric ideal.
For the latter purpose, we exploit the existing software for finding the Gr\"obner fan of toric ideals such as TiGERS or CATS (incorporated in GFAN) (see \cite{HuberThomas2000, Jensen1, Gfan} ).  This software works with homogeneous ideals and in Section \ref{sect:exploiting} we explain how to adapt it to our situation by taking as input the   homogeneized  ideal  and modifying the output in a suitable way. The algorithm  produced in this way  turns out to be much more efficient than that presented in \S \ref{sec:cercafan}. Section \ref{sec:examples} is devoted to present some examples.  Section \ref{sec:problems} points out some open problems and directions for future work.

\section{Preliminaries and notation}\label{sez:preliminari}
Let $A\in\mathbf{M}(d,m;\Z)$ be a $d\times m$ integer matrix (along the paper, $d$ will be replaced by either $r$ (the Picard number) or $n$ (the dimension), depending on the situation). Then
\begin{eqnarray*}
  &\mathcal{L}_{\mathrm r}(A)\subseteq\Z^m& \text{denotes the sublattice spanned by the rows of $A$;} \\
  &\mathcal{L}_{\mathrm c}(A)\subseteq\Z^d& \text{denotes the sublattice spanned by the columns of $A$;} \\
  &\mathcal{V}_{\mathrm r}(A)\subseteq\R^m& \text{denotes the subspace spanned by the rows of $A$;} \\
   &\mathcal{V}_{\mathrm c}(A)\subseteq\R^d& \text{denotes the subspace spanned by the columns of $A$;} \\
  &A_I\,,\,A^I& \text{$\forall\,I\subseteq\{1,\ldots,m\}$ the former is the submatrix of $A$ given by}\\
  && \text{the columns indexed by $I$ and the latter is the submatrix of}\\
  && \text{$A$ whose columns are indexed by the complementary }\\
  && \text{subset $\{1,\ldots,m\}\backslash I$;} \\
\end{eqnarray*}

 We denote by
${\rm supp}(\mathbf{u})$ the support $\{i \ |\  u_i\not=0\}$ of a vector $\mathbf{u}\in\R^m$.
 For a monomial $\x
^\mathbf{u} $
we set ${\rm supp}(\x^\mathbf{u}) = {\rm supp}(\mathbf{u})$.\\

For every vector $\u\in\Z^m$ we write $\u=\u^+-\u^-$ where $\u^+,\u^-\in\N^m$ have disjoint support (we assume, here and elsewhere in the paper, that $0$ belongs to the set of natural numbers $\N$). We denote by $g_{\u}$ the binomial $\x^{\u^+}-\x^{\u^-}$.\\

For a subset $X$ of $\R^n$ the \emph{relative interior}, $\Relint(X)$, is the interior of $X$ inside the
affine span of $X$.\\

In the following $V=(\v_1,\ldots, \v_m)$ is a $n\times m$ complete CF--fan matrix (see the following Definition~\ref{def:Fmatrix}), $Q$ is a Gale dual matrix of $V$ (see the following \S~\ref{ssez:Gale}); it has order  $r\times m$, where $r+n=m$.\\

We shall denote $\langle\v_1,\ldots, \v_h\rangle$ the cone spanned by the vectors $\v_1,\ldots, \v_h$. If $A$ is a matrix, then $\langle A\rangle$ denotes the cone spanned by the columns of $A$.\\

If $M$ is a module over a ring $R$, we shall denote by $M^{\vee}$ its dual module: $M^\vee=\Hom_R(M,R)$.\\

If $\sigma$ is a cone in a real vector space $\V$, $\sigma^*$ denotes its dual cone:
$$\sigma^*=\{f\in\V^\vee\ |\ f(\y) \geq 0\hbox{ for every } \y\in\sigma\}.$$

$\K$ is a field of characteristic $0$.

\section{F--matrices}
 An $n$--dimensional $\QQ$--factorial complete toric variety $X=X(\Si)$ of rank $r$ is the toric variety defined by an $n$--dimensional \emph{simplicial} and \emph{complete} fan $\Si$ such that $|\Si(1)|=m=n+r$.\\
		In particular the rank $r$ coincides with the Picard number i.e. $r=\rk(\Pic(X))$.
		\vspace{0.5 cm}
		
		Given such a fan $\Sigma$, it gives rise to
		a matrix $V$ whose columns are    integral vectors generating the rays of the $1$-skeleton $\Sigma(1)$.

		 The matrix $V$ will be called a \emph{fan matrix} of $\Sigma$.
		 \vspace{0.5 cm}
		
		 Fan matrices motivate the following definition (see \cite{RT-LA&GD} for any further detail).
\begin{definition}\label{def:Fmatrix} A (reduced) \emph{F--matrix} is a $n\times m$ matrix  $V$ with integer entries, satisfying the conditions:
		\begin{itemize}
			\item[a)] $\rk(V)=n$;
			\item[b)] all the columns of $V$ are non zero;
			\item[c)] if ${\bf  v}$ is a column of $V$, then $V$ does not contain another column of the form $\lambda  {\bf  v}$ where $\lambda>0$ is a real number.
				\item[d)] $V$ is F--complete: the cone generated by its columns is $\RR^n$;
			\item[e)] the $\gcd$ of the elements of every columns is 1
			
		\end{itemize}
		A \emph{CF--matrix} is a F--matrix satisfying the further requirement
		\begin{itemize}
			\item[f)] ${\mathcal L}_{\mathrm c}(V)=\Z^n$.
		\end{itemize}
	\end{definition}
Let $V=(\v_1,\ldots,\v_m)$ be a $n\times m$ F--matrix.
	Let $\SF(V)$ be the set of all simplicial and complete fans whose 1-skeleton is given by all the rays generated by the columns of $V$. For any choice $\Si\in\SF(V)$ we get a $\QQ$-factorial complete toric variety $X=X(\Si)$.
	\\
	$X(\Si) $ is called a \emph{poly weighted projective space} (PWS)
	if $V$ is a CF--matrix.
	\begin{theorem}\cite[Theorem 2.2]{RT-QUOT}
		Every $\QQ$-factorial complete toric variety $X=X(\Sigma)$ admits a canonical finite abelian covering $$Y(\widehat{\Si})\twoheadrightarrow X(\Si),$$ unramified in codimension $1$, such that $Y(\widehat{\Si})$ is a $PWS$.
	\end{theorem}
	The fans associated to the two varieties have the same combinatoric structure, in the sense that they involve the same sets of indices of columns in the corresponding fan matrices.

\section{The \lq\lq cercafan\rq\rq algorithm for calculating $\SF(V)$ and the pseudofan conjecture }\label{sec:cercafan}
Let $V=(\v_1,\ldots,\v_m)$ be a $n\times m$ F--matrix. Then  $\SF(V)$ can be computed by the following steps.
\begin{enumerate}
	\item Compute the set $\mathcal{M}$ of minimal $n$-dimensional  cones generated by the columns of $V$, i.e. the simplicial cones not contaning other columns of $V$ than those generating their extremal rays.
	\item Compute the set $\mathcal{F}$ of facets (that is, $(n-1)$-dimensional faces) of cones in $\mathcal{M}$ and for each $f\in\mathcal{F}$ a normal vector $\n_f$ (the \emph{positive side} of $f$).
	\item for every $f\in\mathcal{F}$ compute
\begin{align*}
\mathcal{M}_f^+&= \{\sigma \in \mathcal{M} \ |\ f\hbox{ is a facet of $\sigma$ and $\sigma$  lies on the positive side of $f$}\}\\
\mathcal{M}_f^-&= \{\sigma \in \mathcal{M} \ |\ f \hbox{ is a facet of $\sigma$ and $\sigma$  lies on the negative side of $f$}\}
\end{align*}
 \item Starting from the collection
$$\{(f, \mathcal{M}_f^+,\mathcal{M}_f^-) \ | \ f\in\mathcal{F}\}$$
eliminate in all possible ways cones from $\mathcal{M}_f^+$ and $\mathcal{M}_f^-$ until for every $f\in\mathcal{F}$\\
either $|\mathcal{M}_f^+|=|\mathcal{M}_f^-|=0$ or $|\mathcal{M}_f^+|=|\mathcal{M}_f^-|=1$.
\item Put $\mathcal{C}=\bigcup_f \mathcal{M}_f^+\cup\bigcup_f \mathcal{M}_f^-$.
\item Verify that $\buildrel\circ\over\sigma\cap \buildrel\circ\over\tau=\emptyset$ for every $\sigma,\tau\in \mathcal{C}$; in this case $\mathcal{C}=\Sigma(n)$ for a complete simplicial fan $\Sigma$.
	\end{enumerate}

\begin{conjecture} Step (6) is unnecessary. \end{conjecture}

More precisely we can define a \emph{pseudofan} as a collection
$$\mathcal{C}=\{\sigma_1,\ldots,\sigma_t\}$$ of simplicial cones on the columns of $V$ satisfying:
\begin{itemize}
	\item $\mathcal{C}\subseteq\mathcal{M}$ i.e. every $\sigma_i$ is a minimal maximal cone.
	\item The set of vertices of cones in $\mathcal{C}$ is the whole set of columns of $V$;
	\item  Every facet $f$ of some $\sigma_i\in\mathcal{C}$ is shared by  unique other $\sigma_j\in\mathcal{C}$ lying on the opposite side of $f$.
	\item (maybe) $\mathcal{C}$ is minimal with these properties.
\end{itemize}

Then the pseudofan conjecture can be stated as:
\begin{conjecture}\label{congettura} Every pseudofan is a fan in $\SF(V)$.
	\end{conjecture}
    The conjecture is trivial in dimension 2. Although we would not able to prove it in dimension $\geq 3$, it is supported by a huge number of examples up to dimension 5 and rank 4.

\section{Gale duality and secondary fan}
Although its theoretical interest, the \lq\lq cercafan \rq\rq algorithm is not very efficient and is in practice unusable as the size and number of vectors grow. The following of this paper will be devoted to present a second algorithm, based on the  theory of Gr\"obner bases of toric ideals, computing the projective fans in $\SF(V)$, for a given F--matrix $V$. A fundamental tool for this construction is  Gale duality that we briefly recall in this section.
\subsection{The Gale dual matrix of $V$}\label{ssez:Gale}

Let $V$ be a $n\times m$ F--matrix. If we think $V$ as a linear application from $\Z^{m}$ to $\Z^n$ then $\ker(V)$ is a lattice in $\Z^{m}$, of rank $r:=m-n$ and without cotorsion. \\
We shall denote $Q={\mathcal G}(V)$ the \emph{Gale dual matrix of $V$}: an integral $r\times (n+ r)$ matrix $Q$ whose rows are a $\ZZ$-basis of $\ker(V)$;
the matrix $Q$ is well-defined up to left multiplication by $\GL_r(\Z)$ and can be characterized by the following universal property  \cite[\S 3.1]{RT-LA&GD}:
\begin{center} {\it if $A\in \mathbf{M}_r(\Z)$ is such that $A\cdot V^T=0$ then $A=\alpha\cdot Q$ for some matrix $\alpha\in\mathbf{M}_r(\Z)$.}\end{center}

F--completeness implies that the span of the rows of $Q$ contain a strictly positive vector, so that \emph{$Q$ can be chosen non-negative}.

Notice that $\G(V)=\G(\widehat{V})$ where $\widehat{V}$ is the CF--matrix associated to the $1$-covering.

\subsection{Gale duality}

Recall notation introduced in \S~\ref{sez:preliminari}.

Submatrices of $Q$ and $V$ correspond each other via the natural isomorphism

$$\Z^{m-k}/\Ls_c(Q^I)\isom \Ls_c(V)/\Ls_c(V_I),\quad\quad |I|=k.$$

\subsection{Bunches and the secondary fan}
The Gale dual analogue of a fan is a \emph{bunch of cones} (see \cite{Berchtold-Hausen}). \\
In our situation we can identify a bunch of cones with  a family $\Omega$ of simplicial cones on the columns of $Q$ satisfying
$$\buildrel\circ\over\sigma \cap \buildrel\circ\over\tau\not= \emptyset\hbox{ for every } \sigma,\tau\in\Omega.$$
A bunch $\Omega$ is \emph{projective} if $\bigcap_{\sigma\in \Omega}\sigma $ (the chamber) is full-dimensional.\\
The set of chambers gives rise to the $r$-skeleton of a polyhedral fan (the \emph{secondary fan}) with support $\langle Q\rangle$.\\
If  $I\subseteq\{1,\ldots,m\}$ and $|I|=r$ then the correspondence
$\langle Q_I \rangle \mapsto \langle V^I\rangle $ induces a bijection between bunches and fans and in particular between
\begin{equation}\label{eq:fanandbunches} \xymatrix{\SF(V)\ar@{<->}[r]^-{1:1}& \mathcal{B}(Q)}\end{equation}
where
$\mathcal{B}(Q)=\{\Omega \ |\ \hbox{ for  $i=1,\ldots, m$ there is $\sigma\in\Omega$  such that $\sigma\subseteq \langle Q^{\{i\}} \rangle$ }\}$.
It induces a bijection
$$\xymatrix{\PSF(V)\ar@{<->}[r]^-{1:1}& \mathcal{PB}(Q)}$$
between projective fans and associated bunches.

\begin{remark}
We implemented the cercafan algorithm in Maple, making it compatible with  packages like \texttt{Convex} \cite{convex} and \texttt{MDSpackage} \cite{MDSpackage}. Then one can quickly determine, for each computed fan, if it is projective, by checking if the intersection of all the cones in the bunch (i.e. the nef cone) is full dimensional or not. Moreover both nef and movable cones are quickly computable and one can also obtain useful information about possible embedded Mori Dream spaces.
\end{remark}

\section{The toric ideal}\label{sez:toricideal}
In this section we shall associate to every F--matrix $V$ a toric ideal in a suitable polynomial ring, and investigate the properties of its  Gr\"obner fan, which provides the structure of the initial ideals.  Whilst most of the literature on this subject  deals with homogeneous toric ideals, ideals associated to F--matrices  cannot be homogeneous, due to the F--completeness requirement in  Definition \ref{def:Fmatrix} d).
 We shall describe the Gr\"obner region and translate the problem of calculating fans in a problem of linear programming, as done in \cite{SturmfelsThomas97} in the homogeneous case.\\
Let $\K$ be a field of characteristic $0$ and
 $V=(\v_1,\ldots,\v_m)$ be an  F--matrix; it defines a ring homomorphism:

\begin{eqnarray*}
	\pi:\K[x_1,\ldots, x_m ]& \longrightarrow &\K[t_1^{\pm 1},\ldots t_n^{\pm1}]\\
	x_j&\longmapsto & \mathbf{t}^{\mathbf{v}_j}
\end{eqnarray*}
The \emph{toric ideal} of $V$, denoted as $I_V$, is the kernel of the map $\pi$.\\
The following facts on $I_V$ (see \cite{Sturmfels1996} ) are well-known
	\begin{itemize}
		\item $I_V=I_{\widehat{V}}$.
		\item $I_V$ is generated by binomials $g_{\u}=\x^{\u^+}-\x^{\u^-}$, with $\u=\u^+-\u^-\in\Ls_r(Q)$, and $\u^+,\u^-\in\N^m$ have disjoint support.
		\item For every term order $\preceq$ the reduced Gr\"obner basis of $I_V$ with respect to $\preceq$ consists of a finite set of binomials of the form $g_\u$.
		
		\end{itemize}
		
		\begin{remark}
			Completeness is equivalent to the fact that for $i=1,\ldots, m$ there is a binomial $\x^{\u^+}-1$ in $I_V$ such that $i\in{\rm supp}(\u^+)$.
			In particular every Gr\"obner basis for $I_V$ contains a binomial $\x^{\u}-1$ where $\u$ is a non negative vector in $\Ls_r(Q)$.
		This fact has been applied in order to compute non negative vectors in lattices (see \cite{PisonCasares2004}).
		\end{remark}
		 For every $\u\in \N^m$, the \emph{fiber} $\mathcal{F}(\u)$ of $\u$ is defined as
		$$\mathcal{F}(\u)=(\mathcal{L}_{\mathrm r}(Q)+\u)\cap \N^m.$$\\
		It is noteworthy to remark that fibers are infinite sets in the complete case.\\

         We shall make use of the following technical result:
 \begin{proposition}\label{prop:FiberDecomposition} Let $f=\sum_{i\in I}\lambda_i\x^{\u_i}\in \K[\x] $ be a polynomial and write $f=f_1+\cdots +f_k$ where each summand $f_i$ is the sum of all monomials of $f$ whose exponent vectors lie in the same fiber. Then $f\in I_V$ if and only if $f_i\in I_V$ for $i=1,\ldots, k$.
 \end{proposition}
 \begin{proof} Obviously, $f\in I_V$ if every $f_i$ lies in $I_V$. Conversely, suppose that $f\in I_V$. Then we can write $f=\sum_{j\in J}\mu_jg_{\u_j}=\sum_{j\in J}\mu_j (\x^{\u_j^+}-\x^{\u_j^-})$ where $\mathcal{F}(\u_j^+)=\mathcal{F}(\u_j^-)$. Then each $f_i$ is in turn a sum of some summands of the form  $\mu_j g_{\u_j}$, so that $f_i\in I_V$.
 \end{proof}
			For $\bb\in\R^n$ we define
			$$P_\bb=\{\u\in \R^m\ |\ V\u=\bb, \u\geq \mathbf{0}\}.$$
			(In the above formula, both $\mathbf{u}$ and $\mathbf{b}$ are thought of as \lq\lq column\rq\rq vectors, in order for matrix multiplication to be defined).
			By the completeness of $V$ we know that for every $\bb$, $P_\bb$ is a strictly convex polyhedron of dimension $r$. Moreover if $\bb\in\Z^n$ then  $P_\bb\cap\Z^m =\F(\u)$ for every $\u\in\N^m$ such that $V\u=\bb$. \begin{remark}
  Notice that many notation introduced in the present section depend implicitly on the choice of a fixed F-matrix $V$, as, for instance, for the fiber $\mathcal{F}$ and the polyhedron $P_\mathbf{b}$ and the following cone $\mathcal{W}$. Anyway, we prefer to keep a lighter notation without explicitly expressing $V$.
\end{remark}
		Let $\w\in\R^m$ and consider the linear map
		\begin{eqnarray*}
			\varphi_\w: \R^m_{\geq 0} &\longrightarrow &\R\\
			\u &\longmapsto & \w^T\u
		\end{eqnarray*}
	
	For $\bb\in \R^n$ put
	$$\mathcal{W}_\bb=\{\w\in \R^m\ |\ \varphi_\w \hbox{ is lower bounded on } P_\bb\}$$
	\begin{proposition}\label{prop:Wbnondipendedab} $\mathcal{W}_\bb$ does not depend on $\bb\in\RR^n$. \end{proposition}
	\begin{proof}
		Let $\w\in \mathcal{W}_\bb$ and suppose $\varphi_\w(\x)\ge K\in \RR$, for any $\x\in P_\bb$.
		Let $\bb'\in\R^n$ and choose by completeness $\x_0\in \R^{m}_{\geq 0}$ such that $V\x_0=\bb-\bb'$. Let $\x\in P_{\bb'}$ then
		$$\varphi_\w(\x)=\w^T(\x+\x_0-\x_0)=\w^T(\x+\x_0)-\w^T\x_0\geq K-\w^T\x_0$$
		so that $\w\in \mathcal{W}_{\bb'}$.
	\end{proof}
		Let
\begin{equation}\label{eq:U}\mathcal{U}=P_\mathbf{0}=\ker(V)\cap \R^m_{\geq 0}=\mathcal{V}_{\mathrm r}(Q) \cap \R^m_{\geq 0}. \end{equation}

    By the previous proposition, for every $\bb\in \R^n$ we have $\mathcal{W}_\bb=\mathcal{W}$ where
	\begin{equation}\label{eq:defW}
		\mathcal{W} = \mathcal{W}_{\mathbf{0}}=
	 \{\w\ |\ \varphi_\w \hbox{ is lower bounded on }\mathcal{U}\}
	\end{equation}

	\begin{proposition}
		$\mathcal{W}$ is a polyhedral convex cone, containing $\R^m_{\geq 0}$.
	\end{proposition}
	\begin{proof}
		If $\w,\w'\in\mathcal{W}$ then there exists $K\in\R$ such that
		$\w^T\x, \w'^T\x\geq K$ for every $\x\in P_{\mathbf 0}$. Then for $\lambda,\mu\geq 0$ and $\x\in P_{\mathbf 0}$
		$$(\lambda \w+\mu \w')^T\x\geq (\lambda+\mu)K$$
		so that $\lambda \w+\mu \w'\in\mathcal{W}$, and $\mathcal{W}$ is a polyhedral convex cone.\\
		The second assertion is obvious, since $P_{\mathbf 0}\subseteq \R^m_{\geq 0}$.
	\end{proof}
	
	\begin{proposition}\label{prop:dualityUW}
		$\mathcal{W}$ is dual to the cone $\mathcal{U}$:
		$$\mathcal{W}=\mathcal{U}^*=\{\w\in \R^m\ |\ \w^T\x\geq 0,\forall \x\in \mathcal{U}\}.$$
	\end{proposition}
	\begin{proof}
		Obviously $\mathcal{U}^*\subseteq \mathcal{W}$. Conversely suppose $\w\in\mathcal{W}$ and $\w^T\x<0$ for some $\x\in\mathcal{U}$. Then $\lim_{\lambda\to +\infty }\w^T(\lambda \x)=-\infty$, so that $\varphi_\w$ is not lower bounded, a contradiction. Therefore $\mathcal{W}\subseteq \mathcal{U}^*$.
	\end{proof}

 The next propositions give further characterizations of the cone $\mathcal{W}$ which will be useful in the following:
\begin{proposition}
	\label{prop:coWtre}
	$$\mathcal{W}=\{\w\in\R^m\ |\ \exists \y\in\R^n \hbox{ such that }  V^T\y\leq \w\}.$$
\end{proposition}
The proof is an immediate consequence of the following variant of Farkas' Lemma applied with $A=V^T$, by observing that $\coU=\{\u\in \R^m\ |\ V\u=\mathbf{0}, \u\geq \mathbf{0}\}$.
\begin{proposition}\label{prop:Farkasvariante}\cite[Corollary 7.1e, p. 89]{Schrijver86}
	Let $A$ be a $m\times n$ real matrix and $\w\in\R^m$ be a column vector. Then the system $A\y\leq \w$ of linear inequalities has a solution $\y\in\R^n$ if and only if $\w^T \u\geq 0$ for each column vector $\u\geq \mathbf{0}$ in $\R^m$ such that $A^T\u=\mathbf{0}$.
\end{proposition}
\begin{proposition}\label{prop:W}
	$\coW=Q^{-1}(\langle Q\rangle)$.
\end{proposition}
\begin{proof}
	We firstly show that $Q^{-1}(\langle Q\rangle)\subseteq \coW$. Let $\w\in Q^{-1}(\langle Q\rangle)$. Then $Q(\w)\in\langle Q\rangle$, so that there esists $\w'\in\R^m_{\geq 0}$ such that $Q(\w)=Q(\w')$. This in turn implies that there exists $\t\in\coL_{\mathrm r}(V)$ such that $\w=\w'+\t$. If $\u\in \coU =\coV_{\mathrm r}(Q)\cap \R^m_{\geq 0}$ then $\w^T \u=(\w')^T\u+\t^T\u\geq 0$, because $(\w')^T\u\geq 0$ and $\t^T\u=0$. Then $\w\in\coU^*=\coW$.\\
	Conversely, let $\u\in Q^{-1}(\langle Q\rangle)^*$, so that
	\begin{equation}
	\label{eq:ddd}
	\t^T\u\geq 0\hbox{ for every } \t\hbox{ such that } Q\t\in\langle Q\rangle.\end{equation}
	Then $\u\in\coV_{\mathrm r}(Q)$; indeed if $\t\in\coV_{\mathrm r}(V)$ then $Q\t=0$ so that, $\pm \t\in Q^{-1}(\langle Q\rangle)$; then (\ref{eq:ddd}) implies that $\pm\t^T\u\geq 0$, so that $\t^T\u=0$. Moreover $\u\in \R^m_{\geq 0}$; indeed every element $\e_i$ in the canonical basis of $\R^m$ lies in $Q^{-1}(\langle Q\rangle)$, so that (\ref{eq:ddd}) implies that $u_i\geq 0$ for $i=1,\ldots, m$. Therefore we have shown that $(Q^{-1}(\langle Q\rangle))^*\subseteq \coV_{\mathrm r}(Q)\cap \R^m_{\geq 0}=\coU$; by the properties of the dual cone this implies $\coW=\coU^*\subseteq (Q^{-1}(\langle Q\rangle))^{**}=Q^{-1}(\langle Q\rangle))$.
\end{proof}
		
        \subsection{The Gr\"obner fan} To every
			$\w\in\R^m$, we can associate a relation $ \preceq_\w\ $  on $\N^m$, defined by  $$\u_1\preceq_\w \u_2\hbox{ if } \varphi_\w(\u_1)\leq \varphi_\w(\u_2).$$
			Let $\KK$ be any field, and $\mathbf{x}= x_1,\ldots, x_m$. For a polynomial $f=\sum_{\aa} c_\aa\mathbf{x}^\aa\in \KK[\mathbf{x}]$ the \emph{inital term} $\mathrm{in}_\w(f)$ of $f$ w.r.t. $\w$ is defined as the sum of all terms $c_\aa\mathbf{x}^\aa$ in $f$ such that $\varphi_\w(\aa)$ is maximal.  If $I$ is an ideal in $\KK[\x]$, the \emph{initial ideal} of $I$ w.r.t. $\w$ is then $\mathrm{in}_\w(I)=\{\mathrm{in}_\w(f)\ |\ f\in I\}$. If $\mathrm{in}_\w(I)$ is monomial, then $\w$ is said \emph{generic} for $I$.
			It is well-known \cite[Cor. 1.10 and Prop. 1.11]{Sturmfels1996} that the set of initial ideals $\mathrm{in}_\preceq(I)$ of $I$ determined by term orders $\preceq$ coincide with the set of  inital  ideals $\mathrm{in}_\w(I)$ of $I$ associated to generic weight vectors  $\w\in\RR^m_{\geq 0}$.  \\
			
When $\w\not\in\R^m_{\geq 0}$, $\preceq_\w$  cannot be refined to be a term order; however it is still possible that, for some ideal $I$, $\mathrm{in}_{\w}(I)$  is also the initial ideal of $I$ with respect to some term order.

 Therefore we introduce the following:
 \begin{definition}
 	Let $I\subseteq \KK[\x]$ be an ideal. A monomial ideal $J$ is an \emph{initial ideal} of $I$ if $J=\mathrm{in}_\preceq(I)$ for some term order $\preceq$.
 \end{definition}

 \begin{proposition}\label{prop:iniIVimplicaW} Let  $\w\in\R^m$ be generic; if
 	$\mathrm{in}_{\w}(I_V)$ is an initial ideal of $I_V$ then $\w\in\mathcal{W}$.\end{proposition}

 \begin{proof}
 	Suppose that $\mathrm{in}_{\w}(I_V)$ is an initial ideal of $I_V$; then by definition there is a term order $\preceq $ such that  $\mathrm{in}_{{\w}}(I_V)=\mathrm{in}_{\preceq}(I_V)$. Let $\u\in\mathcal{L}_{\mathrm r}(Q)\cap\R^m_{\geq 0}$; then $\x^\u-1\in I_V $ and   thus $\x^\u\in \mathrm{in}_\preceq(I_V)=\mathrm{in}_{{\w}}(I_V)$. Then $\x^\u\succeq_\w 1$ so that $\w^T\u\geq 0$. Notice that by \eqref{eq:U} the cone $\mathcal{U}$ is generated by  $\mathcal{L}_{\mathrm r}(Q)\cap\R^m_{\geq 0} $.  This shows that $\w\in \mathcal{U}^*=\mathcal{W}$, by Proposition \ref{prop:dualityUW}.
 \end{proof}

 Two vectors in $\mathcal{W}$ determine the same initial ideal of $I_V$ when they represent linearly (i.e. numerically) equivalent divisors; this is established by the following

 \begin{proposition}
 	Choose $\w_1,\w_2 \in\mathcal{W}$. If $ Q\w_1=Q\w_2$ then
 	$\mathrm{in}_{{\w_1}}(I_V)=\mathrm{in}_{{\w_2}}(I_V)$.
 \end{proposition}
 \begin{proof}
 	We can write $\w_2=\w_1+\r$, with $\r\in\ker(Q)$. Then  $\w_1^T \u =\w_2^T\u$ holds for every $\u\in\mathcal{L}_{\mathrm r}(Q)$, so that $\mathrm{in}_{{\w_1}}(g_\u)=\mathrm{in}_{{\w_2}}(g_\u)$.  Then the result follows because $I_V$ is generated by binomials $g_{\u}$ for $\u\in \Ls_r(Q)$.
 \end{proof}
 The converse of Proposition \ref{prop:iniIVimplicaW} is also true. In order to establish it we need some preliminary results.

 \begin{lemma}\label{lem:raffinamento}
 	Let $\u\in\N^m, \w\in \mathcal{W}$. Then $\preceq_\w$  refines the standard partial order $\leq$ on $\mathcal{F}(\u)$;  that is
 	$$\v_1\leq \v_2\Rightarrow {\v_1}\preceq_\w {\v_2}, \hbox{ for every } \v_1,\v_2\in\mathcal{F}(\u).$$
 	
 \end{lemma}
 \begin{proof} If $\v_1\leq \v_2$ then $\v_2=\v_1+\t$ for some $\t\in\R^m_{\geq 0}\cap \mathcal{L}_{\mathrm r}(Q)\subseteq\coU$. Since $\w\in\coW$, then $\w^T\t\geq 0$ by Proposition \ref{prop:dualityUW}, so that $\w^T\v_1\leq \w^T \v_2$.
 \end{proof}

 Let us recall the following well known result:
 \begin{lemma}[Dickson Lemma, \cite{HerzogHibi2010} Theorem 2.1.1]
 	Every totally unordered subset of $\N^m$ is finite.
 \end{lemma}

 \begin{lemma}
 	\label{lem:minimo} Let $\w\in \mathcal{W}$ be generic. For every $\u\in\N^m$ the fiber $\mathcal{F}(\u)$ has a (unique) minimum with respect to $\preceq_\w$.
 \end{lemma}
 \begin{proof} By Proposition \ref{prop:Wbnondipendedab} and formula \eqref{eq:defW}, $\varphi_\w$ is lower bounded on $\mathcal{F}(\u)$.
 	Therefore the set
 	$$S:=\{\v\in \mathcal{F}(\u)\ |\ \varphi_\w \hbox{ attains its mimimum at  } \v\}$$
 is not empty. Assume that S contains two distinct elements $\v_1\not= \v_2$. Then $\x^{\v_1}-\x^{\v_2}\in I_V$ and $\varphi_\w(\v_1)=\varphi_\w(\v_2)$, so that $\x^{\v_1}-\x^{\v_2}\in \mathrm{in}_\w(I_V)$. Since $\w$ is generic, $\mathrm{in}_\w(I_V)$ is monomial; thus  $\x^{\v_1}, \x^{\v_2}\in  \mathrm{in}_\w(I_V)$.  Since $\x^{\v_1}\not\in I_V$, by Proposition \ref{prop:FiberDecomposition}  there exists $\v_0\in\mathcal{F}(\u)$ such that $\varphi_\w(\v_0)<\varphi_\w(\v_1)$, contrarily to the fact that  $\v_1\in S$. Then $S$ contains a single element.
 \end{proof}

 For a generic $\w\in\mathcal{W}$ define
 \begin{eqnarray*}
 	\Min_\w &=& \{\u\in\N^m \ |\ \u\hbox{ is the minimum of } \mathcal{F}(\u)\hbox{ w.r.t. }\preceq_\w\}\\
 	\overline{\Min}_\w &=& \N^m\setminus \Min_\w\\
 \end{eqnarray*}
 and let $\MnM_\w$ be the subset of $\nM_\w$ consisting of the minimal elements with respect to $\leq$. By Dickson's Lemma, $\MnM_\w$ is a finite set.

 \begin{lemma}\label{lem:supporto}  If $\u\in\MnM_\w$ and $\u_0$ is the minimum of $\mathcal{F}(\u)$ w.r.t. $\preceq_{\w}$, then ${\rm supp}(\u)\cap {\rm supp}(\u_0)=\emptyset$.
 \end{lemma}
 \begin{proof} Assume that $j\in {\rm supp}(\u)\cap {\rm supp}(\u_0)$, let $\e_j$ be the vector having  $1$ at place $j$ and $0$ elsewhere and put $\u'=\u-\e_j$, $\u'_0=\u_0-\e_j$. Then $ \u'_0\preceq_\w\u' <\u$ and $\u'-\u'_0=\u-\u_0$, so that $\u'$ and $\u'_0$ lie in the same fiber. Therefore $\u'\in\nM_\w$ and $\u$ is not minimal.
 \end{proof}
 \begin{lemma}\label{lem:iniziale} Let $\w\in\mathcal{W}$ be generic. Then the following equality of ideals holds:
 	$$ \mathrm{in}_{\w}(I_V)=( \x^\u\ |\ \u \in \nM_\w ) =(\x^\u\ |\ \u \in \MnM_\w)$$
 \end{lemma}
 \begin{proof}
 	The second equality is trivial, so we prove the first one. For $\u\in \nM_\w$, let $\u_0$ be the minimum of $\mathcal{F}(\u)$ with respect to $\preceq_\w$. Then $\u-\u_0\in\mathcal{L}_{\mathrm r}(Q)$, and ${\rm supp}(\u)\cap{\rm supp}(\u_0)=\emptyset$ by Lemma \ref{lem:supporto}. Then by definition $\x^\u-\x^{\u_0}\in I_V$, so that $x^\u\in \mathrm{in}_{\preceq_\w}(I_V)$. Therefore the inclusion $(x^\u\ |\ \u \in \nM_\w )\subseteq \mathrm{in}_{\w}(I_V)$ is shown. If it were not an equality, $\mathrm{in}_{\w}(I_V)$ would contain a monomial $\x^{\u_0}$ for some $\u_0\in \Min_\w$. Therefore it would exist a polynomial $f\in I_V$ such that $\mathrm{in}_{\preceq_\w}(f)=\x^{\u_0}$; by Proposition \ref{prop:FiberDecomposition} we can assume that all monomials appearing in $f$ lie in the same fiber, and since ${\u_0}$ is the minimum of its fiber $f$ must  be a single term: $f=\lambda\x^{\u_0}$ for some $\lambda\in \KK$. But $\pi(\x^{\u_0})$ is a monomial in $\K[\t^{\pm 1}]$, so it cannot be zero.
 \end{proof}

 With the notation above, let $\MnM_\w=\{\aa_1,\ldots,\aa_s\}$; for every $i$ let $\mathbf{m}_i$ be the minimum of $\mathcal{F}(\aa_i)$ with respect to $\preceq_\w$, and define $\bb_i=\aa_i-\m_i$. Notice that, by Lemma \ref{lem:supporto}, $\aa_i$ and $\m_i$ have disjoint support, so that $\aa_i=\bb_i^+$ and $\m_i=\bb_i^-$, and $g_{\bb_i}= \x^{\aa_i}-\x^{\m_i}$. Put $\mathcal{B}_\w=\{\bb_1,\ldots, \bb_s \}$ and $\mathcal{G}_\w=\{g_{\bb_i} \ |\ i=1,\ldots, s\}\subseteq I_V$.

 \begin{lemma}[Farkas Lemma, \cite{Ziegler95} Proposition 1.7]
 Let $U$ be a $d\times m$ real matrix, and $\z\in\R^d$. Then one and only one of the following holds:
 \begin{itemize}
   \item[-] there exists a vector $\y\in\R^m$  with $U\y \leq \z$,
   \item[-] there exists a vector $\u\in\R^d$ with $\u \geq 0$, $\u^T U = 0$ and $\u^T\z<0$.
 \end{itemize}
 \end{lemma}

 \begin{proposition}\label{prop:iniziale2}
 	Let $\w\in\mathcal{W}$; then there is $\mathbf{c}\in\R^m_{\geq 0}$ such that $\mathrm{in}_{\preceq_\w}(I_V)=\mathrm{in}_{\preceq_\mathbf{c}}(I_V)$.
 \end{proposition}
 \begin{proof}   Let $B$ be the $s\times m$ matrix having rows $\bb_1,\ldots, \bb_s$ and let $U$ be the $(m+s)\times m$ matrix defined by
 	$$U=\begin{pmatrix} I_m\\
 	B\end{pmatrix}$$
 	where $I_m$ is the square identity matrix of order $m$. We firstly show that there exists  $\cc\in \R^m_{\geq 0}$  such that $\cc\cdot \bb_i > 0$ for every $i=1,\ldots, s$. Suppose the contrary. Then there does not exist $\y\in\R^m$ such that $(-U)\y\leq \z$ where $\z=(-1,\ldots, -1)^T$.  There exists $\u\in\R^{m+s}$ such that $\u\geq 0$, $\u\not=\mathbf{0}$ and $\u^TU=\mathbf{0}$. Write $\u=(u^*_1,\ldots,u^*_m,u_1,\ldots, u_s)$. Then for $j=1,\ldots, m$
 	$$u^*_j+u_1b_{1j}+\ldots +u_sb_{sj}=0$$
 	which implies
 	$$u_1b_{1j}+\ldots +u_sb_{sj}\leq \mathbf{0}\hbox{ for } j=1,\ldots, m.$$
 	It follows that $$u_1\bb_1+\ldots +u_s\bb_s\leq \mathbf{0},\quad\hbox{ with } u_1,\ldots, u_s\geq 0, \hbox{ not all zero }. $$
 	By Lemma \ref{lem:raffinamento} this implies that $\sum_{i=1}^s u_i\bb_i\preceq_\w 0$. On the other hand by definition we have $\bb_i\succ_\w \mathbf{0}$ and $u_1,\ldots, u_s$ are not all zero, therefore $\sum_{i=1}^s u_i\bb_i\succ_\w \mathbf{0}$, a contradition.\\
 	It follows that $\mathrm{in}_{\w}(I_V)\subseteq \mathrm{in}_{\mathbf{c}}(I_V)$. If this inclusion was not an equality then by Lemma \ref{lem:iniziale} $\mathrm{in}_{\mathbf{c}}(I_V)$ should contain some element $\m$ in $\text{Min}_\w$ and so the  whole fiber $\mathcal{F}(\m)$, contradicting Lemma \ref{lem:iniziale} itself.
 	
 \end{proof}
 By collecting Proposition \ref{prop:iniIVimplicaW} and Lemma \ref{prop:iniziale2} we obtain the following
 \begin{theorem}\label{prop:iniIVsseW} Let $\w\in\R^m$ be generic; then
 	$\mathrm{in}_{\w}(I_V)$ is an initial ideal of $I_V$ if and only if $\w\in\mathcal{W}$.\end{theorem}

Theorem \ref{prop:iniIVsseW}  asserts that $\mathcal{W}$ is the \emph{Gr\"obner region} of $I_V$, in the sense of \cite{MoraRobbiano1988}. It is the support of the \emph{Gr\"obner fan} of $I_V$, whose construction we  sketch briefly in the following.
		Two weight vectors $\w_1,\w_2\in  \mathcal{W}$ are said \emph{equivalent} modulo $I_V$ if $\mathrm{in}_{\w_1}(I_V)=\mathrm{in}_{\w_2}(I_V)$.
				Equivalence classes form relatively open polyhedral cones in $\RR^m$, whose closures are said \emph{Gr\"obner cones}. For $\w\in\coW$ let $\mathcal{C}[\w]$ be the smallest Gr\"obner cone containing $\w$. Then $\w$ is generic if and only if $\mathcal{C}[\w]$ is full-dimensional.
				The Gr\"obner fan of $I_V$ is the collection of Gr\"obner cones $\mathcal{C}[\w]$, for $\w\in \mathcal{W}$.
				\subsection{The Gr\"obner fan and the secondary fan}
				Obviously $\w+\ker Q\subseteq \mathcal{C}[\w]$ for every $\w\in\mathcal{W}$, so that we can consider the image fan via $Q$ of the Gr\"obner fan in $\R^r$.\\
				In this way the Gr\"obner cone and the secondary fan  both live in $\R^r$ and they have the same support $\langle Q\rangle$, by Proposition \ref{prop:W}.\\
					The following is a crucial result by Sturmfels:
					\begin{theorem}\cite[Proposition 8.15]{Sturmfels1996} The Gr\"obner fan refines the secondary fan.
						\end{theorem}

By restricting 	the support of both the Gr\"obner fan and the secondary fan to $\Mov(Q)=\bigcap_{i=1}^m \langle Q^{\{i\}}\rangle$ we get, in the light of \eqref{eq:fanandbunches}, the following result

\begin{corollary}\label{cor:suriezione}
						There is a surjective computable map
						$$\left\{ \hbox{Initial ideals of $I_V$ not containing a power of $x_i$, $\forall i=1,\ldots, m$}\right\}\twoheadrightarrow \PSF(V).$$
						\end{corollary}
					\begin{proof}
					
					We follow the argument proving \cite[Proposition 8.15]{Sturmfels1996}. Given  an initial ideal $\mathcal{I}$ for $I_V$ the corresponding fan can be computed as
					\begin{align*} \Sigma_\mathcal{I}&=\{\langle V^I\rangle \ |\ I\subseteq \{1,\dots m\} \hbox{ and } {\rm supp}(\u)\cap I\not=\emptyset \hbox{ for some monomial } \u \in \mathcal{I}\}\\
					&=\{\langle V_J\rangle \ |\ J\subseteq \{1,\dots m\} \hbox{ and } {\rm supp}(\u)\not\subseteq  J \hbox{ for every monomial } \u \in \mathcal{I}\}
					\end{align*}
Then we see that for $i=1,\ldots, m$
\begin{align*} \langle \v_i\rangle \in \Sigma_\mathcal{I}(1) &\Leftrightarrow {\rm supp}(\u)\not\subseteq  \{i\}  \hbox{ for every monomial } \u \in \mathcal{I}\\
&\Leftrightarrow \mathcal{I} \hbox{ does not contain a power of $x_i$}.
\end{align*}
Conversely, from $\Sigma$ we can recover the radical $\sqrt{\mathcal{I}}$ as the \emph{Stanley-Reisner ideal} $\Delta(\Theta)$ of the simplicial complex $\Theta$ associated to $\Sigma$:
					\begin{align*}
					\Theta &= \{ J\subseteq \{1,\ldots, m\}\ |\ \langle V_J\rangle \in\Sigma\}\\
						\sqrt{\mathcal{I}} &= \Delta(\Theta)= (\prod_{i\in I} x_i\ |\ I\in\{1,\ldots, m\} \hbox{ and } I\not\in\Theta\}.
					\end{align*}
                    This shows that the correspondence $\mathcal{I}\mapsto \Sigma_\mathcal{I}$ is surjective.\end{proof}
					
In general the correspondence $\mathcal{I}\mapsto \Sigma_\mathcal{I}$ is not injective; see Example \ref{ex:raffinamentoproprio} below. It is shown in \cite[Corollary 8.9]{Sturmfels1996} that $\mathcal{I}$ is radical if and only if $\Sigma_\mathcal{I}$  gives rise to a smooth toric variety i.e. all cones in $\Sigma_\mathcal{I}(n)$ have normalized volume 1. It follows that the correspondence is injective when $V$ is unimodular (see \cite[Remark 8.10]{Sturmfels1996}, taking into account that $V$ is a CF--matrix,  so that its maximal minors are coprime).
					
					\section{Calculating $\PSF(V)$}\label{sec:algorithm}
                 Corollary \ref{cor:suriezione} provides an algorithm for computing the set $\PSF(V)$ for any F--matrix $V$:
                 \subsection{The \lq\lq $G$-cercafan\rq\rq\  algorithm \label{subsect:G-cercafan}}
					\begin{enumerate}
					\item Compute the toric ideal $I_V$.
                    \item Compute the Gr\"obner fan of $I_V$.
                    \item For any full-dimensional Gr\"obner cone $\mathcal{C}[\w]$ in the Gr\"obner fan, compute the initial ideal $\mathcal{I}_\w=\mathrm{in}_\w(I_V)$; \item Eliminate those initial ideals containing a power of some $X^i$.
                    \item For the remaining initial ideals $\mathcal{I}$, compute the fan $\Sigma_\mathcal{I}$ defined in the proof of Corollary \ref{cor:suriezione}.
                    \item Remove duplicate fans (if any).
                    \end{enumerate}

                    \begin{remark}  Steps (2) and (3) may be englobed: in fact the software we used produces directly the initial ideals without making use of weight vectors. \end{remark}

\subsection{Exploiting the existing software}\label{sect:exploiting}
Initial ideals can be determined by  software computing the Gr\"obner fan of toric ideals, such as TiGERS \cite{TiGERS} or CATS  \cite{Jensen1} (incorporatd in GFAN \cite{Gfan}); we mainly used the last one.\\
As well explained in \cite{HuberThomas2000}, the general idea is starting with an arbitrary term order and generate successively all the initial ideals by going across \lq\lq facet binomials \rq\rq\  in the corresponding Gr\"obner bases (flips).

There are two main strategies: for small size inputs, an exhaustive search is possible. When the latter is impracticable, a \lq\lq Gr\"obner walk\rq\rq\  can be implemented, progressively producing new Gr\"obner cones and consequently new fans.\\

This software always requires homogeneous toric ideals as input.
Since  toric ideals associated to F--matrices are never homogeneous, we adapted it to our situation by the following procedure:
\begin{itemize}
	\item Homogenize $I_V$ by adding an auxiliary variable $x_{m+1}$, getting a homogeneous toric ideal $HI_V$.
	\item Compute the set of initial ideals of $HI_V$.
	\item Eliminate those initial ideals containing a power of $x_{m+1}$.
	\item Dehomogenize the remaining initial ideals with respect to $x_{m+1}$
    \item For any monomial ideal so obtained, pick up a minimal set of generators.
\end{itemize}
It is possible to show that this procedure is correct, i.e. it produces the Gr\"obner fan for $I_V$.\\

In conclusion, we get an algorithm which is much more efficient than \lq\lq cercafan\rq\rq; however it presents two main disadvantages:
\begin{itemize}
	\item the same fan is computed many times;
	\item non-projective fans are not seen by Gr\"obner methods.
\end{itemize}
\section{Some examples}\label{sec:examples}

\subsection{An example of proper refinement}\label{ex:raffinamentoproprio}
Let $n=4, r=3$ and
$$Q=\begin{pmatrix} 1& 1&0& 0& 2& 0&0\\
0& 0& 1& 1& 2& 0&0\\ 0& 0& 0& 0& 1& 1&1\end{pmatrix}, \quad V=\G(Q)=\begin{pmatrix}
1& 1&0& 2&-1&0&1\\
0&2&0&2&-1&0&1\\
0&0&1&-1&0&0&0\\
0&0&0&0&0&1&-1
\end{pmatrix}$$

The toric ideal is
$$I_V=(x_1x_2-x_3x_4, x_5x_6x_7-1, x_3x_4x_5-x_6x_7, x_6^2x_7^2-x_3x_4).$$

$\SF(V)$ contains three fans, and all of them are projective.

There are six reduced Grobner bases. Every chamber in the secondary fan consists of two Gr\"obner cones. For example the following are two Gr\"obner bases whose initial ideals have the same radical:
\begin{eqnarray*}
	GB_1 &=&
	(x_1x_2-x_3x_4,
	x_3x_4x_5-x_6x_7,
	x_5x_6x_7-1, x_6^2x_7^2-x_3x_4)
	\\ In_1&=&(x_1x_2,x_3x_4x_5,x_5x_6x_7, x_6^2x_7^2 ) \\
	GB_2&=& (x_1x_2-x_3x_4,
	x_6x_7-x_3x_4x_5,
	x_3x_4x_5^2-1)\\
	In_2&=&(x_1x_2,x_6x_7,x_3x_4x_5^2 )
\end{eqnarray*}
\begin{figure}
\begin{center}
\includegraphics[width=8truecm]{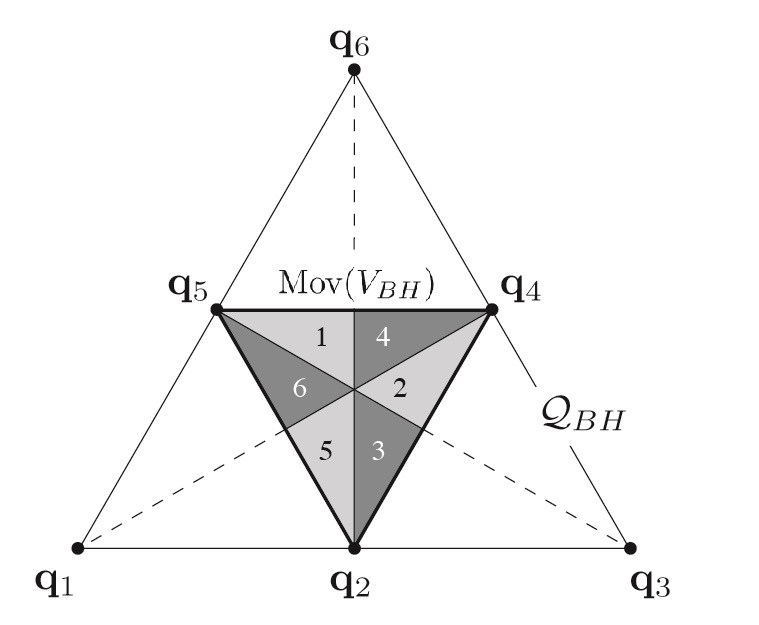}
\caption{\label{fig1}The section, by the standard simplex in $\RR^3$, of the effective cone $\gkz_{BH}:=\langle Q_{BH}\rangle$ and the movable cone $\Mov(V_{BH})$ in \cite[Ex.\,10.2]{Berchtold-Hausen}.}
\end{center}
\end{figure}
\subsection{A non-projective fan}\label{ssez:example}
Let us consider the well-known Berchtold-Hausen example \cite[Ex.\,10.2]{Berchtold-Hausen}, whose weight and fan matrices can be presented as follows
$$Q_{\text{BH}}=\begin{pmatrix} 1& 1&0& 0& 1& 0\\
0& 1& 1& 1& 0& 0\\ 0& 0& 0& 1& 1& 1\end{pmatrix}\quad
V_{\text{BH}}=\G\left(Q_{\text{BH}}\right)=\begin{pmatrix}
1& 0&0& 0&-1&1\\
0&1&0&-1&-1&2\\
0&0&1&-1&0&1
\end{pmatrix}$$
Figure \ref{fig1} represents a section of the effective cone $\mathcal{Q}_{\text{BH}}:=\langle Q_{\text{BH}}\rangle\subset\R^3$ with the standard simplex in $\R^3$, well describing the secondary fan.

The toric ideal is
$$I_{V_{\text{BH}}}=(x_1x_5-x_3x_4, x_2x_3-x_5x_6, x_1x_2-x_4x_6, x_4x_5x_6-1, x_1x_5^2x_6-x_3).$$

 $\SF(V_{\text{BH}})$ contains eight fans, but GFAN returns only six  inital ideals having different radicals: there are two non-projective fans. The six chambers enumerated in figure \ref{fig1} give nef cones of the associated projective varieties. The two non-projective varieties shares the same nef cone, given by the intersection of all the six chambers and generated by the anti-canonical class.

\section{Open problems}\label{sec:problems}

\subsection{Recovering the initial ideal from the fan} For fans in $\PSF(V)$ whose inverse image, by the map in Corollary \ref{cor:suriezione}, is a singleton, it is possible to recover, from the fan, a set of generators of the  initial ideal, by choosing, in a universal Gr\"obner basis $\mathcal{U}_V$ for $I_V$, those binomials having a monomial support in the fan and considering, in each of them, the complementary monomial. This procedure can be performed also when the inverse image contains more than one initial ideal, or even when the fan is not projective. Of course,  in the latter case, the obtained monomial ideal will not be an initial ideal of $I_V$. However one can ask if it may be in some way significant and providing some interesting  informations about the original fan. Anyway, this question seems to be related to the more general one: are non projective fans be detectable by algebraic methods?

\subsection{The irrelevant ideal and Alexander duality} When passing from the initial ideal to the fan via the map of Corollary \ref{cor:suriezione}, one loses the information given by exponents in the generators of the ideal. This happens because the fan is only determined by the radical of the initial ideal, and in general initial ideals of $I_V$ are not radical. Recall that the \emph{irrelevant ideal} of a toric variety is the reduced monomial ideal having a  minimal generator for each maximal cone in the fan, namely the product of the variables not indexed by elements of the list corresponding to the maximal cone. For complete $\Q$-factorial toric variety, the irrelevant ideal is the Alexander dual of the Stanley-Reisner ideal of the fan.  In \cite{Miller1998,Miller2000}, Miller provides a construction of Alexander duality  for monomial ideals which may also be not radical. It would be interesting to study  Alexander duals of initial ideals of $I_V$  in the non-radical case, and investigate  which kind of information they carry on about the geometry of the corresponding toric varieties.

\subsection{The state polyhedron} It is well-known that, for homogeneous toric ideals, the Gr\"obner region coincides with the whole $\RR^m$, so that the Gr\"obner fan results to be polytopal and initial ideals can be recovered as vertices in the \emph{state polytope} associated to $I_V$ (see \cite[Chapter 2]{Sturmfels1996}). Of course in our case the Gr\"obner fan is not polytopal, since fibers are infinite sets.  Nevertheless, the Minkowski sum of Gr\"obner fibers give rise to a poyhedron see \cite[Sections 2-4]{Babson2003}, the \emph{state polyhedron} of $I_V$, whose normal fan coincides with the Gr\"obner fan of $I_V$.
The lower boundary of the state polyhedron should be the analog of the state polytope in the complete case. We intend to deepen this investigation in a future work.

\subsection{Deforming the secondary fan}\label{ssez:deforming}

Consider the Berchtold-Hausen example presented in the previous \S~\ref{ssez:example} and move slightly a generator of the movable cone $\Mov(V_{\text{BH}})$, e.g. $\q_4$, to a near rational vector still belonging to the cone generated by the nearest generators of the effective cone $\langle Q_{\text{BH}}\rangle$, that is $\langle \q_3,\q_6\rangle$. This produces a deformation of the secondary fan represented in figure \ref{fig2} by choosing, e.g.,
\begin{equation}\label{q4}
  \q_4:=\left(
     \begin{array}{c}
       0\\
       1\\
       2\\
     \end{array}
   \right)
\end{equation}
\begin{figure}
\begin{center}
\includegraphics[width=7truecm]{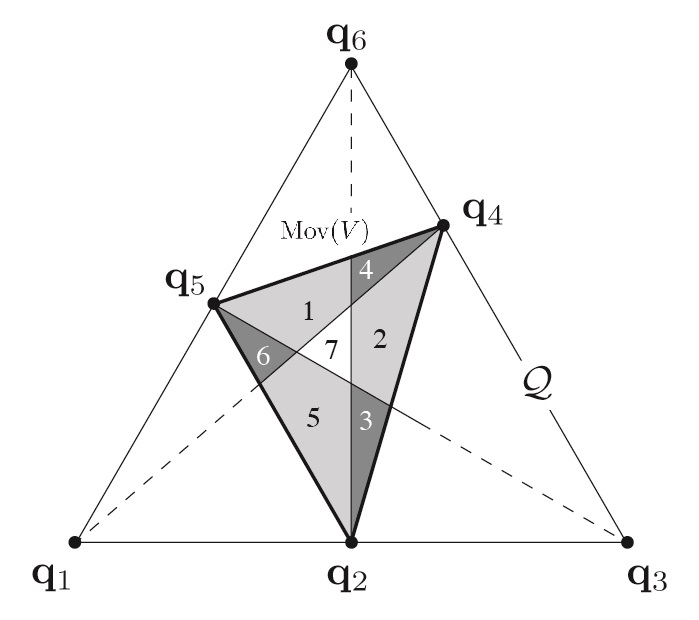}
\caption{\label{fig2}The section, by the standard simplex in $\RR^3$, of the effective cone $\gkz=\langle Q\rangle$ and the movable cone $\Mov(V)$, in the example studied in \S~\ref{ssez:deforming}.}
\end{center}
\end{figure}
\noindent that is
\begin{equation*}
Q:=\left(
     \begin{array}{cccccc}
       1&1&0&0&1&0\\
      0&1&1&1&0&0\\
       0&0&0&2&1&1 \\
     \end{array}
   \right)\ ,\quad V=\G(Q)=\left(
  \begin{array}{cccccc}
    1&0&0&0&-1&1 \\
    0&1&0&-1&-1&3 \\
    0&0&1&-1&0&2 \\
  \end{array}
\right)\,.
\end{equation*}
This example has been extensively studied by the authors in \cite[\S~3]{RT-r2proj}, where we listed all the 8 fans in $\SF(V)$ and in $\SF(V_{\text{BH}})$, by running  the Cercafan algorithm, implemented as a Maple routine. Fans in $\SF(V)$ are exactly the same of those in $\SF(V_{\text{BH}})$ after replacing $\v_6=(1\,,\,2\,,\,1)^T$ with $\v_6=(1\,,\,3\,,\,2)^T$. In particular the 6 projective varieties assigned by chambers from 1 to 6 in figure \ref{fig1} remain projective and assigned by chambers in figure 2 admitting the same number. Moreover the 2 non-projective fans in $\SF(V_{\text{BH}})$ sharing the same nef cone given by $\langle(1\,,\,1\,,\,1)^T\rangle$, and represented by the barycenter of the triangle in figure \ref{fig1}, give rise to a new projective fan represented by chamber 7 in figure \ref{fig2} and to a non-projective one, whose associated toric variety has trivial nef cone.

The just described situation, determined by the particular choice (\ref{q4}), is actually the generic situation when $\q_4$ runs over all possible rational vectors in the open cone $\Relint\langle \q_3,\q_6\rangle$. Such a variation is given by
$$\forall\,p,q\in\N\setminus\{0\}:(p,q)=1\quad\q_4=\left(
     \begin{array}{c}
       0\\
       q\\
       p\\
     \end{array}
   \right)\in\Relint\left\langle\begin{array}{cc}
       0&0\\
       1&0\\
       0&1\\
     \end{array}\right\rangle$$
determining the following variations of weight and fan matrices:
\begin{eqnarray*}
Q_{p/q}&:=&\left(
     \begin{array}{cccccc}
       1&1&0&0&1&0\\
      0&1&1&q&0&0\\
       0&0&0&p&1&1 \\
     \end{array}
   \right)\\
   V_{p/q}=\G(Q_{p/q})&=&\left(
  \begin{array}{cccccc}
    1&0&0&0&-1&1 \\
    0&1&q-1&-1&-1&p+1 \\
    0&0&q&-1&0&p \\
  \end{array}
\right)
\end{eqnarray*}

Considering only the 7-th chamber, gives rise to a family $\mathcal{X}^7\longrightarrow \Q$ of complete and $\Q$-factorial toric varieties  whose generic fibre $X^7_{p/q}$ is projective and admitting a special non projective fibre $X^7_1$.

On the other hand, considering the 8-th chamber gives rise to a family $\mathcal{X}^8\longrightarrow \Q$ of never projective, complete, $\Q$-factorial toric varieties, whose generic fibre $X^8_{p/q}$ does not admit any non-trivial nef divisor and admitting a special fibre $X^8_1$ endowed with a non trivial nef divisor (in the present example represented by the anti-canonical one).

Finally chambers from 1 to 6 give rise to 6 families $\mathcal{X}^i\longrightarrow B$ of projective $\Q$-factorial toric varieties.

We believe these are interesting phenomena which should often appear in studying toric deformations of toric varieties, but we are not aware of any theoretical explanation clarifying the described situation. An implementation of algorithms here presented can certainly produce a number of evidences over which trying to understand something more.


\begin{thebibliography}{10}

\bibitem{Babson2003}
{\sc Babson, E., Onn, S., and Thomas, R.}
\newblock The {H}ilbert zonotope and a polynomial time algorithm for universal
  {G}r\"{o}bner bases.
\newblock {\em Adv. Appl. Math. 30}, 3 (2003), 529--544.

\bibitem{Berchtold-Hausen}
{\sc Berchtold, F., and Hausen, J.}
\newblock Bunches of cones in the divisor class group---a new combinatorial
  language for toric varieties.
\newblock {\em Int. Math. Res. Not.}, 6 (2004), 261--302.

\bibitem{DeLoera1996}
{\sc de~Loera, J.~A., Hosten, S., Santos, F., and Sturmfels, B.}
\newblock The polytope of all triangulations of a point configuration.
\newblock {\em Doc. Math. 1\/} (1996), No. 04, 103--119.

\bibitem{DeLoera2010}
{\sc De~Loera, J.~A., Rambau, J., and Santos, F.}
\newblock {\em Triangulations}, vol.~25 of {\em Algorithms and Computation in
  Mathematics}.
\newblock Springer-Verlag, Berlin, 2010.
\newblock Structures for algorithms and applications.

\bibitem{Demazure}
{\sc Demazure, M.}
\newblock Sous--groupes alg\'ebriques de rang maximum du groupe de {C}remona.
\newblock {\em Ann. Sci. \'Ecole Norm. Sup. 3\/} (1970), 507--588.

\bibitem{convex}
{\sc Franz, M.}
\newblock {\em Convex - a Maple package for convex geometry (version 1.2)},
  2016.
\newblock Available at \url{http://www.math.uwo.ca/faculty/franz/convex/}.

\bibitem{FP}
{\sc Fujino, O., and Payne, S.}
\newblock Smooth complete toric varieties with no nontrivial nef line bundles.
\newblock {\em Proc. Japan. Acad. 81\/} (2005), 174--179.

\bibitem{MDSpackage}
{\sc Hausen, J., and Keicher, S.}
\newblock A software package for {M}ori dream spaces.
\newblock {\em LMS J. Comput. Math. 18}, 1 (2015), 647--659.

\bibitem{HerzogHibi2010}
{\sc Herzog, J., and Hibi, T.}
\newblock {\em Monomial Ideals}.
\newblock Springer-Verlag, New York-Heidelberg, 2010.
\newblock Graduate Texts in Mathematic, No. 260.

\bibitem{TiGERS}
{\sc Huber, B.}
\newblock {\em TiGERS}.
\newblock Available at
  \url{https://sites.math.washington.edu/~thomas/programs/tigers.html}.

\bibitem{HuberThomas2000}
{\sc Huber, B., and Thomas, R.~R.}
\newblock Computing {G}r\"{o}bner fans of toric ideals.
\newblock {\em Experiment. Math. 9}, 3 (2000), 321--331.

\bibitem{ImaiMasada2002}
{\sc Imai, H., Masada, T., Takeuchi, F., and Imai, K.}
\newblock Enumerating triangulations in general dimensions.
\newblock {\em Internat. J. Comput. Geom. Appl. 12}, 6 (2002), 455--480.

\bibitem{Gfan}
{\sc Jensen, A.~N.}
\newblock {G}fan, a software system for {G}r{\"o}bner fans and tropical
  varieties.
\newblock Available at
  \url{http://home.imf.au.dk/jensen/software/gfan/gfan.html}.

\bibitem{Jensen1}
{\sc Jensen, A., N.}
\newblock {\em CaTS, a software package for computing state polytopes of toric
  ideals}.
\newblock Available at \url{http://www.soopadoopa.dk/anders/cats/cats.html}.

\bibitem{Kleinschmidt-Sturmfels}
{\sc Kleinschmidt, P., and Sturmfels, B.}
\newblock Smooth toric varieties with small {P}icard number are projective.
\newblock {\em Topology 30}, 2 (1991), 289--299.

\bibitem{Miller1998}
{\sc Miller, E.}
\newblock Alexander duality for monomial ideals and their resolutions.
\newblock {\em arXiv:math/9812095\/} (1998).

\bibitem{Miller2000}
{\sc Miller, E.}
\newblock The {A}lexander duality functors and local duality with monomial
  support.
\newblock {\em Journal of Algebra 231}, 1 (2000), 180 -- 234.

\bibitem{MoraRobbiano1988}
{\sc Mora, T., and Robbiano, L.}
\newblock The {G}röbner fan of an ideal.
\newblock {\em Journal of Symbolic Computation 6}, 2 (1988), 183 -- 208.

\bibitem{MO}
{\sc Oda, T.}
\newblock {\em Torus embeddings and applications}, vol.~57 of {\em Tata
  Institute of {F}und. {R}esearch}.
\newblock Springer--Verlag, Berlin -- New York, 1978.
\newblock ({B}ased on a joint work with {K}.\,{M}iyake).

\bibitem{Oda}
{\sc Oda, T.}
\newblock {\em Convex bodies and algebraic geometry}, vol.~15 of {\em
  Ergebnisse der Mathematik und ihrer Grenzgebiete (3) [Results in Mathematics
  and Related Areas (3)]}.
\newblock Springer-Verlag, Berlin, 1988.
\newblock An introduction to the theory of toric varieties, Translated from the
  Japanese.

\bibitem{PisonCasares2004}
{\sc Pisón-Casares, P., and Vigneron-Tenorio, A.}
\newblock N-solutions to linear systems over z.
\newblock {\em Linear Algebra and its Applications 384\/} (2004), 135 -- 154.

\bibitem{RT-LA&GD}
{\sc Rossi, M., and Terracini, L.}
\newblock $\mathbb{Z}$--linear gale duality and poly weighted spaces {(PWS)}.
\newblock {\em Linear Algebra Appl. 495\/} (2016), 256--288.

\bibitem{RT-QUOT}
{\sc Rossi, M., and Terracini, L.}
\newblock A $\mathbb{Q}$--factorial complete toric variety is a quotient of a
  poly weighted space.
\newblock {\em Ann. Mat. Pur. Appl. 196\/} (2017), 325--347; {\urlfont
  arXiv:1502.00879}.

\bibitem{RT-r2proj}
{\sc Rossi, M., and Terracini, L.}
\newblock A $\mathbb{Q}$--factorial complete toric variety with {Picard} number
  2 is projective.
\newblock {\em J. Pure Appl. Algebra 222}, 9 (2018), 2648--2656.

\bibitem{Schrijver86}
{\sc Schrijver, A.}
\newblock {\em Theory of linear and integer programming}.
\newblock Wiley-Interscience Series in Discrete Mathematics. John Wiley \&
  Sons, Ltd., Chichester, 1986.

\bibitem{Sturmfels1996}
{\sc Sturmfels, B.}
\newblock {\em Gr\"{o}bner bases and convex polytopes}, vol.~8 of {\em
  University Lecture Series}.
\newblock American Mathematical Society, Providence, RI, 1996.

\bibitem{SturmfelsThomas97}
{\sc Sturmfels, B., and Thomas, R.}
\newblock Variation of cost functions in integer programming.
\newblock {\em Mathematical Programming 77\/} (1997), 357--387.

\bibitem{Ziegler95}
{\sc Ziegler, G.}
\newblock {\em Lectures on polytopes}, vol.~152 of {\em Graduate Texts in
  Mathematics}.

\end{thebibliography}
\end{document}